\documentclass[11pt,a4paper]{amsart}
\usepackage{amsmath,mathtools,mathrsfs,multirow}

\calclayout
\DeclareFontEncoding{LS1}{}{}
\DeclareFontSubstitution{LS1}{stix2}{m}{n}
\DeclareMathAlphabet{\mathscr}{LS1}{stix2scr}{m}{n}

\usepackage{stmaryrd}
\usepackage{comment}
\usepackage{graphicx}
\usepackage{amsfonts}
\usepackage{amssymb}
\usepackage{wrapfig,lipsum,booktabs}
\usepackage{placeins}
\usepackage[hang]{subfigure}
\usepackage[utf8]{inputenc}
\usepackage[english]{babel}
\usepackage{amsthm}
\usepackage{subfigure} 
\usepackage[utf8]{inputenc}
\usepackage[margin=1in]{geometry} 
\usepackage[final]{pdfpages}
\newtheorem{theorem}{Theorem}
\newtheorem*{theorem*}{Theorem}
\newtheorem*{remark}{Remark}
\newtheorem{lemma}{Lemma}
\theoremstyle{definition}

\theoremstyle{notation}

\theoremstyle{corollary}
\newtheorem*{corollary}{Corollary}

\newtheorem*{conj}{Conjecture}

\usepackage{times}
\usepackage{blindtext}
\usepackage{setspace}
\usepackage{hyperref}
\usepackage{cleveref}

\author{Yazan Alamoudi}
\address{Department of Mathematics\\
University of Florida\\
Gainesville\\
FL 32611\\
United States
}
\email{yazanalamoudi@ufl.edu}

\title{On a nonnegativity conjecture of Andrews}

\begin{document}

\maketitle
\begin{abstract}
I settle a conjecture of Andrews related to the Alladi-Schur polynomials. In addition, I give further relations and implications to two families of polynomials related to the Alladi-Schur polynomials.
\end{abstract}

\section{Introduction}
The conjecture in question relates to polynomials arising in an investigation of integer partitions. That investigation \cite{And1} is what led to Andrews' refinement\footnote{The Alladi-Schur theorem is the statement that the number of partitions of $n$ into odd parts, each occurring at most twice, is equal to the number of Schur partitions of $n$. This was first observed by K. Alladi and communicated to G. Andrews (see \cite{And1}).} of the Alladi-Schur theorem, which is the equation \[|\mathcal{C}(m,n)|=|\mathcal{D}(m,n)|\tag{1.1},\] where $\mathcal{C}(m,n)$ is the set of partitions of $n$ into $m$ odd parts, each occurring at most twice, and $\mathcal{D}(m,n)$ is the set of Schur partitions $\pi$ of $n$ where the number of parts plus the number of even parts of $\pi$ is $m$. Recall that Schur partitions are partitions into parts that differ by at least $3$ with no consecutive multiples of $3$. Of relevance is that Andrews proved (1.1) using some recursive relations (see \cite{And1} and \cite{And2}) of the Alladi-Schur polynomials  \[d_N(x)=\sum_{m,n\geq0}D_N(m,n)x^mq^n,\]
where $D_N(m,n)=|\mathcal{D}_N(m,n)|$ and $\mathcal{D}_N(m,n)$ is the set of partitions in $\mathcal{D}(m,n)$ with parts $\leq N$. More relevant are those relations Andrews gives in \cite{And2}, such as $p_n(x)\mid d_{6n-1}(x)$ and its explicit form
\[\frac{d_{6n-1}(x)}{p_n(x)}=\sum_{i=0}^nc(n,i)x^i\tag{1.2},\]
where $p_n(x)=\prod_{i=1}^n(1+xq^{2i-1}+x^2q^{4i-2})$. Indeed, Andrews' conjecture can now be presented. 
\begin{conj}[Andrews' conjecture stated in \cite{And2}]
For all $n$ and $j$, $c(n,j)$ has nonnegative coefficients.    
\end{conj}
 I resolve the above conjecture by a very similar approach to the one Andrews used to prove his factorization theorem for $d_n(x)$. The aforementioned theorem \cite[Thm 1.2]{And2} states that $p_{\lfloor\frac{n+4}{6}\rfloor}(x)\mid d_n(x)$ for $n\not\equiv 3\pmod6$, and $p_{\lfloor\frac{n-2}{6}\rfloor}(x)\mid d_n(x)$ otherwise.
It is easy to see that this result can be written as  \[p_{\lceil\frac{n+3\chi_o(n)}{6}\rceil -\chi_o(n)}(x)\mid d_n(x),\]
where $\chi_o$ is the indicator function for odd integers. To this end, define\footnote{In \cite{And2} Andrews has instead written $\Delta(i,n,x,q)$,  with $i\in\{-1,0,1,2,3,4\}$, in place of what I write as $\mathscr{d}_{6n-i}(x)$.} the polynomials $\mathscr{d}_n$ as \[\mathscr{d}_n(x)=\frac{d_n(x)}{p_{\lceil\frac{n+3\chi_o(n)}{6}\rceil-\chi_o(n)}(x)}.\tag{1.3}\]
In due course, I establish the validity of Andrews' conjecture by proving the following stronger claim.

\begin{theorem}
    For $n\geq1$, $\mathscr{d}_{n}$ is a polynomial in $x$ and $q$ with nonnegative integer coefficients. Hence, Andrews' conjecture is true.  
\end{theorem}

After establishing Andrews' conjecture, this paper concludes its technical matter with a section dedicated to extending some of the ideas contributing to the proof of Theorem 1 to an appropriate level of generality. This includes recursive relations for $\mathscr{d}_{n}(x)$, which can be thought of as the corresponding relations to Andrews' recursive relations of the Alladi-Schur polynomials $d_n(x)$. In addition, explicit formulas for $\mathscr{d}_n(x)$ in terms of $x,q,$ and $c(n,j)$ are given.\footnote{Such formulas were alluded to by Andrews in \cite{And2} , but they were not explicitly stated.} Lastly, this final chapter also mentions some implications for $c(n,j)$ that show, for example, why Theorem 1 is stronger than the statement of Andrews' conjecture.

\section{Proof of Theorem 1}
We begin with the following useful lemma, which already gives one recursive relation for the polynomials $\mathscr{d}_{n}(x)$.

\begin{lemma}
For $N\geq3$, we have the following.
    \[\mathscr{d}_{2N}(x)=\mathscr{d}_{2N-3}(xq^2)\tag{2.1}\]
\end{lemma}
\begin{proof}
Recall that in \cite{And2}, Andrews has shown that, for $N\geq3$, we have $d_{2N}(x)=\lambda(x)d_{2N-3}(xq^2)$ with ${\lambda(x)=1+xq+x^2q^2}$. It follows that,
\[
    \mathscr{d}_{2N}(x)=\frac{d_{2N}(x)}{p_{\lceil\frac{2N}{6}\rceil}(x)}=\frac{\lambda(x)d_{2N-3}(xq^2)}{p_{\lceil\frac{(2N-3)+3}{6}\rceil}(x)}=\frac{d_{2N-3}(xq^2)}{p_{\lceil\frac{(2N-3)+3}{6}\rceil-1}(xq^2)}=\mathscr{d}_{2N-3}(xq^2).\]
This establishes the lemma.    
\end{proof}

To settle Andrews' conjecture, we need one more lemma. First, observe that, by the standard technique of removing the largest part (see \cite{YA} and also \cite{And1, And2}), it is easy to see that
\[
 d_N(x)=d_{N-1}(x)+x^{1+\chi_2(N)}q^{N}d_{N-3-\chi_3(N)}(x)\tag{2.2},   \]
where $\chi_n$ denotes the indicator function for the set of multiples of $n$. The next lemma gives the analogous property for only the odd-indexed $\mathscr{d}_n(x)$, which is the case that will be used in the proof of Theorem 1, with the even-indexed case being presented in the next section.

\begin{lemma}
    For $2N-1\geq 5$ with $2N-1\not\equiv0\pmod{3}$ we have
    \[
        \mathscr{d}_{2N-1}(x)=\mathscr{d}_{2N-2}(x)+xq^{2N-1}\mathscr{d}_{2N-4}(x)\tag{2.3}.
    \]
    Otherwise, for $N\geq0$, with initial\footnote{This can be arrived at from Andrews' convention that $d_{-1}(x)=1$ (see \cite{And1,And2}) and the convention that an empty product $=1$.} condition $\mathscr{d}_{-1}(x)=1$, we have
    \[
      \mathscr{d}_{6N+3}(x)=\lambda_{N+1}(x)\mathscr{d}_{6N+2}(x)+xq^{6N+3}\mathscr{d}_{6N-1}(x)\tag{2.4}, 
    \]
    where $\lambda_N(x)=(1+xq^{2N-1}+x^2q^{4N-2})$.
\end{lemma}
\begin{proof}
For odd indices not divisible by $3$, either\footnote{For ease, one may use $\lceil\frac{n+3\chi_o(n)}{6}\rceil -\chi_o(n)=\lceil\frac{n-3\chi_o(n)}{6}\rceil$ throughout, but I opted to avoid any confusion from $\lceil\frac{-2}{6}\rceil= 0$.}
    \[\mathscr{d}_{6N+1}(x)=\frac{d_{6N+1}(x)}{p_N(x)}=\frac{d_{6N}(x)+xq^{6N+1}d_{6N-2}(x)}{p_N(x)}=\mathscr{d}_{6N}(x)+xq^{6N+1}\mathscr{d}_{6N-2}(x),\]
or    
    \[\mathscr{d}_{6N-1}(x)=\frac{d_{6N-1}(x)}{p_N(x)}=\frac{d_{6N-2}(x)+xq^{6N-1}d_{6N-4}(x)}{p_N(x)}=\mathscr{d}_{6N-2}(x)+xq^{6N-1}\mathscr{d}_{6N-4}(x).\]
Otherwise, if the index is divisible by three, then\smallskip
    \[\mathscr{d}_{6N+3}(x)=\frac{d_{6N+3}(x)}{p_{N}(x)}=\frac{d_{6N+2}(x)+xq^{6N+3}d_{6N-1}(x)}{p_{N}(x)}=\lambda_{N+1}(x)\mathscr{d}_{6N+2}(x)+xq^{6N+3}\mathscr{d}_{6N-1}(x).\]
\end{proof}
We are now ready to prove Theorem 1 and settle Andrews' conjecture. The proof below closely draws on the approach Andrews took in establishing his factorization theorem in \cite{And2}. However, in this manuscript, the problems and results are framed from a new perspective. This plays an essential role in the approach taken here to resolve the conjecture.

\begin{proof}[Proof of Theorem 1]
We proceed by strong induction. For the base case, observe the following.
\begingroup
\setlength{\jot}{6pt}
\begin{gather*}
    \mathscr{d}_{1}(x)=1+xq\\
    \mathscr{d}_{2}(x)=1\\
    \mathscr{d}_{3}(x)=1+x(q+q^3)+x^2q^2\\
    \mathscr{d}_{4}(x)=1+xq^3\\
    \mathscr{d}_{5}(x)=1+x(q^3+q^5)\\
    \mathscr{d}_{6}(x)=1+x(q^3+q^5)+x^2q^6
\end{gather*}
\endgroup
Suppose this is true for every $6\leq n< N'$. For the inductive step, we first handle the case of even indices. If $N'=2N\geq6$, then by Lemma 1  
\[\mathscr{d}_{2N}(x)=\mathscr{d}_{2N-3}(xq^2),\] with  $\mathscr{d}_{2N-3}(xq^2)$ a polynomial in $x$ and $q$ with nonnegative integer coefficients by the induction hypothesis. It follows that $\mathscr{d}_{2N}(x)$ is a polynomial in $x$ and $q$ with nonnegative integer coefficients.\medskip
    
For the case of odd indices, we handle two instances: one where the index is not a multiple of three and another where it is. If $N'= 2N-1\not\equiv0\pmod{3}$, then by (2.3)
\[\mathscr{d}_{2N-1}(x)=\mathscr{d}_{2N-2}(x)+xq^{2N-1}\mathscr{d}_{2N-4}(x).\]
We see that the right-hand side consists of polynomials in $x$ and $q$ with nonnegative integer coefficients. It follows that the same is true for the left-hand side.\medskip

Lastly, for odd indices that are divisible by $3$, write $N'=6N+3$, we have       
      \[\mathscr{d}_{6N+3}(x)=\lambda_{N+1}(x)\mathscr{d}_{6N+2}(x)+xq^{6N+3}\mathscr{d}_{6N-1}(x).\]
Again, the right-hand side consists of polynomials in $x$ and $q$ with nonnegative integer coefficients, and so the same must be true for the left-hand side.\medskip

Now, Andrews' conjecture immediately follows as what was shown here, along with the initial conditions for $c(n,i)$ (e.g., $c(0,0)=1$) presented in \cite{And2}, implies that, in particular, $\mathscr{d}_{6n-1}(x)=\sum_{i=0}^nc(n,i)x^i$ is a polynomial in $x$ and $q$ with nonnegative integer coefficients and $c(n, i)$ is a polynomial in $q$ with nonnegative integer coefficients for any $i,n\in\mathbb{N}$. 

\end{proof}

To better grasp the significance of the result, the reader is invited to examine Andrews' explicit formula for $c(n,j)$ presented in \cite[Thm 1.3]{And2}. From that perspective, the number of years the conjecture remained open can be better appreciated. 
\section{Further relations and remarks related to the quotients $\mathscr{d}_{n}(x)$}

We begin by completing the analogy of equation (2.2) for the even cases.
\begin{lemma} For $N\geq 1$, if $N\not\equiv 0\pmod{3}$ then
\[
   \lambda_{\lceil\frac{2N}{6}\rceil}(x)\mathscr{d}_{2N}(x)=\mathscr{d}_{2N-1}(x)+x^2q^{2N}\mathscr{d}_{2N-3}(x).\tag{3.1}\]
Otherwise,
   \[\mathscr{d}_{6N}(x)=\mathscr{d}_{6N-1}(x)+x^2q^{6N}\mathscr{d}_{6N-4}(x).\tag{3.2}\]
\end{lemma}
\begin{proof}
For the case where the index is not divisible by three, either 
$$\lambda_{N}(x)\mathscr{d}_{6N-2}(x)=\frac{\lambda_{N}(x)d_{6N-2}(x)}{p_N(x)}=\frac{d_{6N-2}(x)}{p_{N-1}(x)}$$
$$=\frac{d_{6N-3}(x)+x^2q^{6N-2}d_{6N-5}(x)}{p_{N-1}(x)}=\mathscr{d}_{6N-3}(x)+x^2q^{6N-2}\mathscr{d}_{6N-5}(x),$$
or
$$\lambda_{N}(x)\mathscr{d}_{6N-4}(x)=\frac{\lambda_{N}(x)d_{6N-4}(x)}{p_N(x)}=\frac{d_{6N-4}(x)}{p_{N-1}(x)}
$$
$$=\frac{d_{6N-5}(x)+x^2q^{6N-4}d_{6N-7}(x)}{p_{N-1}(x)}
=\mathscr{d}_{6N-5}(x)+x^2q^{6N-4}\mathscr{d}_{6N-7}(x).$$
This proves equation (3.1).
Otherwise,
\[\mathscr{d}_{6N}(x)=\frac{d_{6N}(x)}{p_{N}(x)}=\frac{d_{6N-1}(x)+x^2q^{6N}d_{6N-4}(x)}{p_{N}(x)}=\mathscr{d}_{6N-1}(x)+x^2q^{6N}\mathscr{d}_{6N-4}(x),\]
which proves equation (3.2).
\end{proof}

\begin{remark}
Allow me to point out an amazing fact. Observe that the recursions for $6N$, $6N-1$ and $6N+1$, namely equations  (3.2) and (2.3), have the same form. On the other hand, for index $6N\pm 2$ the recursion is (3.1). Lastly, the recursion for indices that are $6N+3$ has a unique form. Now, in my bijective proof in \cite{YA}, I have introduced the set\footnote{It is to be noted that $\mathcal{D}^*$ is also the set of Schur partitions where the smallest power of $2$ that could occur as a part is $2^3=8$. } $\mathcal{D}^*$ of Schur partitions with odd parts $\geq 3$ and even parts $\geq 6$ and defined the \textit{refined upper factorization} of $\pi$ as a partition of the parts of $\pi\in\mathcal{D}^*$ into intervals called refined factors. Let the $\pi\in\mathcal{D}^*$ have largest part $\lambda$. If $\lambda=6N+3$, then it \textbf{must} be a free odd upper singleton. Moreover, if the index is an odd non-multiple of three, it can only be a member of a pair or a free odd singleton; likewise, when $\lambda=6N>0$, although it cannot be part of a pair, it can be one of two things: a free upper singleton or the largest upper minimal segment singleton. Lastly, if $\lambda$ is an even non-multiple of three, then it could be either a free singleton, a member of an upper pair or the largest part of an upper minimal segment. Observe that we again get the same grouping. From here, define ${\#:\mathbb{N}^+\setminus\{4,2,1\}\to\{0,1,2,3\}}$ by letting $\#(N)$ denote the cardinality of the list of possible kinds of factors that can have $N$ as their largest part. Let $\text{sign}:\mathbb{Q}[x,q]\to\{-1,0,1\}$ be defined by $\text{sign}(0)=0$ and for $\sum\mu_{i,j}x^iq^j\in\mathbb{Q}[x,q]\setminus\{0\}$ by $\frac{\min_{\mu_{i,j}\not=0}\{\mu_{i,j}\}}{|\min_{\mu_{i,j}\not=0}\{\mu_{i,j}\}|}$ (alternatively, $\text{sign}$ assigns $1$ to a non-zero polynomial $f$ with nonnegative coefficients and assigns $-1$ to a non-zero polynomial $g$ if any of its coefficients are negative).
In addition, let ${r_N=r_N(x,q)=\mathscr{d}_N(x)-\mathscr{d}_{N-1}(x)-x^{1+\chi_2(N)}q^{N}\mathscr{d}_{N-3-\chi_3(N)}(x)}$.\\ Then, for $m,n>2$ \[\#(m)\leq\#(n) \Leftrightarrow \text{sign}(r_n)\leq\text{sign}(r_m).\]
Moreover, for\footnote{The relations can be extended in view of $\#(n)=\#(n+6k)$ and $ \text{sign}(r_n)=\text{sign}(r_{n+6k})$. However, I did not find such an extension compelling.} $n>2$ we have the ``conserved quantity" given by \(\#(n)+\text{sign}(r_n)=2.\)
\end{remark}

\begin{theorem}
 For $N\geq3$, we have the following.
\[
    \mathscr{d}_{2N-1}(x)=\lambda_{\lceil \frac{2N+2}{6}\rceil}(x)\mathscr{d}_{2N-4}(xq^2)+xq^{2N-1}(1-xq)\mathscr{d}_{2N-7}(xq^2)\tag{3.3}
\]

\end{theorem}

\begin{proof}

It is possible to do this in a manner very similar to what was done in \cite{YA} for the odd-indexed Alladi-Schur polynomials. For example, if the subscript is $2N-1\not\equiv 0 \pmod{3}$ then
\begin{align*}
\mathscr{d}_{2N-1}(x)&=\mathscr{d}_{2N-2}(x)+xq^{2N-1}\mathscr{d}_{2N-4}(x)\\
&=\mathscr{d}_{2N-5}(xq^2)
+xq^{2N-1}\mathscr{d}_{2N-7}(xq^2)\\
&=\lambda_{\lceil \frac{2N-4}{6}\rceil}(xq^2)\mathscr{d}_{2N-4}(xq^2)+xq^{2N-1}\mathscr{d}_{2N-7}(xq^2)
-x^2q^{2N}\mathscr{d}_{2N-7}(xq^2)\\
&=\lambda_{\lceil \frac{2N+2}{6}\rceil}(x)\mathscr{d}_{2N-4}(xq^2)+xq^{2N-1}(1-xq)\mathscr{d}_{2N-7}(xq^2).
\end{align*}

But it is, in fact, much simpler to start with Andrews' recursive relations for the odd subscript Alladi-Schur polynomials. Namely, for $N\geq3$ Andrews' relation \cite{And2} states 
$ d_{2N-1}(x)=\lambda(x)(d_{2N-4}(xq^2)
+xq^{2N-1}(1-xq)d_{2N-7}(xq^2)).$ Dividing both sides by $p_{\lceil\frac{2N-4}{6}\rceil}(x)$ gives the theorem.
\end{proof}
In what follows, we directly consider, for each $n\geq1$, the coefficients of 
$\mathscr{d}_n(x)$ when written as a polynomial in $x$. Specifically, define $\mathscr{c}(n,i)$ in accordance to
\[
    \mathscr{d}_n(x)=\sum_{i\geq0}\mathscr{c}(n,i)x^i\tag{3.4}.
\]
Notice that the coefficients inherit analogous recurrences to those of $\mathscr{d}_n$. In particular, the relation (2.1) translates to $\mathscr{c}(2N,i)=\mathscr{c}(2N-3,i)q^{2i}$.\medskip

The upcoming theorem gives formulas of $\mathscr{c}(n,i)$ with $n\geq1$ in terms of $c(n',i')$. In view of Andrews' explicit formula for $c(n',i')$, found in \cite{And2}, these amount to explicit formulas $\mathscr{c}(n,i)$ as polynomials in $q$.

\begin{theorem}
For $N\geq1$, we have the following.

\begin{align*}
\mathscr{c}(6N\ ,j)\quad\ \ &  =c(N,j)+q^{6N}c(N-1,j-2)q^{2(j-2)}\tag{3.5}\\
\mathscr{c}(6N-1,j)&=c(N,j)\tag{3.6}\\
\mathscr{c}(6N-2,j)&=c(N,j)-q^{6N-1}c(N-1,j-1)q^{2(j-1)}\tag{3.7}\\
\mathscr{c}(6N-3,j)&=q^{-2j}c(N,j)+q^{6N-4}c(N-1,j-2)\tag{3.8}\\
\mathscr{c}(6N-4,j)&=c(N-1,j)q^{2j}\tag{3.9}\\
\mathscr{c}(6N-5,j)&=q^{-2j}c(N,j)-q^{6N-3}c(N-1,j-1)\tag{3.10}\\
\end{align*}
\end{theorem}
\begin{proof}
In view of equation (2.1), we can separate into cases modulo $3$. For index congruent to $2\pmod{3}$ this follows from 
\[\mathscr{d}_{6N-4}(x)=\mathscr{d}_{6N-7}(xq^2)=\sum_{j=0}^{N-1}c(N-1,j)q^{2j}x^j.\]
On the other hand, when the index is congruent to $0\pmod{3}$, the result follows from
\[\mathscr{d}_{6N}(x)=\mathscr{d}_{6N-1}(x)+x^2q^{6N}\mathscr{d}_{6N-7}(xq^2)=\sum_{j=0}^Nc(N,j)x^i+x^2q^{6N}\sum_{i=0}^{N-1}c(N-1,j)q^{2j}x^j\]
\[=\sum_{j=0}^{N+1}(c(N,j)+q^{6N}c(N-1,j-2)q^{2(j-2)})x^j=\mathscr{d}_{6N-3}(xq^2).\]
Lastly, for index congruent to $1\pmod{3}$ we have
\[\mathscr{d}_{6N-2}(x)=\mathscr{d}_{6N-1}(x)-xq^{6N-1}\mathscr{d}_{6N-7}(xq^2)=\sum_{j=0}^Nc(N,j)x^i-xq^{6N-1}\sum_{i=0}^{N-1}c(N-1,j)q^{2j}x^j\]
\[=\sum_{j=0}^N(c(N,j)-q^{6N-1}c(N-1,j-1)q^{2(j-1)})x^j=\mathscr{d}_{6N-5}(xq^2).\]
This completes the proof.
\end{proof}
The corollary below illustrates how Theorem 1 can give more information about $c(n,i)$ than the statements of Andrews' conjecture. 
\begin{corollary}

     For $0<j\leq n$, we have \[
        c(n,j)\geq q^{6n-1}c(n-1,j-1)q^{2(j-1)}\tag{3.11},
    \]
    \[
        c(n,j)\geq \frac{q^{6jn+5j-2j^2-4n-2}(1-q^{4(n-j+1)})}{1-q^2}\tag{3.12},\]
    and
    \[q^{(2n+1)j}\mid c(n,j)\tag{3.13}, \]    
    where for $\alpha,\beta\in\mathbb{R}(q)$ we write $\alpha=\sum a_iq^i\geq\sum b_iq^i=\beta$ if and only if 
 $\,\forall i (a_i\geq b_i)$.
\end{corollary}
\begin{proof}
 In view of Theorem 1 and (3.7), we directly get \[
        c(n,j)\geq q^{6n-1}c(n-1,j-1)q^{2(j-1)}
    .\]
    To obtain (3.12), we will employ one more relation. Now, in \cite{And2}, Andrews deduced \[c(n,1)=\frac{q^{2n+1}(1-q^{4n})}{1-q^2}\tag{3.14}.\]
    We note that we can actually arrive at (3.14) combinatorially. The coefficient of $x$ in $d_{6n-1}$ is the generating function for partitions into one odd part $\leq6n-1$. In view of (1.2) and (1.3), this is also equal to the coefficient of $x$ of $p_n$ plus the coefficient of $x$ in $\mathscr{d}_{6n-1}$. Thus,
    \[\frac{q-q^{6n+1}}{1-q^2}=\frac{q-q^{2n+1}}{1-q^2}+c(n,1).\]
    This implies (3.14). Now, by repeatedly applying (3.11), we get
    \[c(n,j)\geq q^{6n-1}c(n-1,j-1)q^{2(j-1)}\geq q^{6n-1+6(n-1)-1}c(n-2,j-2)q^{2(j-1)+2(j-2)}\geq\dots\]\[\geq q^{6(T_{j-1}+(j-1)(n-j+1))-(j-1)}c(n-j+1,1)q^{2T_{j-1}}=q^{6jn +7j -2j^2 -6n -5}c(n-j+1,1).\]
    Invoking (3.14) gives
    \[c(n,j)\geq q^{6jn +7j -2j^2 -6n -5}c(n-j+1,1)=\frac{q^{6jn+5j-2j^2-4n-2}(1-q^{4(n-j+1)})}{1-q^2}.\]
        
    This completes the proof of (3.12).

    To establish\footnote{Observe that from Theorem 1 and (3.8), it follows that for $0<j\leq n$ we have $q^{2j}\mid c(n,j)$ which, although easy to prove, is weaker than (3.13).} (3.13) and complete the corollary, we first recall the following identity of Andrews (see \cite[Lem. 4.1]{And2}). 
    \[c(n,j)=q^{4j}(c(n-1,j)+(q^{2n-3} +q^{6n-2j-3})c(n-1,j-1) + (q^{4n-6} - q^{6n-2j-4})c(n - 1, j - 2))\tag{3.15}\]
    We remark that (3.15) can be deduced\footnote{In \cite{And2} Andrews establishes (3.15) by a very similar argument.} by first noting that from (3.3) and (2.1), we get, as a special case
    \[\mathscr{d}_{6n-1}(x)=\lambda_{n+1}(x)\mathscr{d}_{6n-7}(xq^4)+xq^{6n-1}(1-xq)\mathscr{d}_{6n-7}(xq^2)\tag{3.16},\]
    after which, rewriting both sides of (3.16) using (1.2) and collecting the coefficient of $x^j$ gives (3.15).
    Now, suppose that (3.13) is true for all $j, n<N$. In such cases (i.e. $j, n<N$), define ${c'(n,j)= \frac{c(n,j)}{q^{(2n+1)j}}}$. Then,\begin{align*}
        c(N,j)&=q^{4j}(c(N-1,j)+(q^{2N-3} +q^{6N-2j-3})c(N-1,j-1) + (q^{4N-6} - q^{6N-2j-4})c(N - 1, j - 2))\\
        &=q^{(2N+1)j}(q^{2j}c'(N-1,j)+(q^{2j-2} +q^{4N-2})c'(N-1,j-1) + (q^{2j-4} - q^{2N-2})c'(N - 1, j - 2)).
        \end{align*}
    The base case is trivial, and the conclusion follows.
\end{proof}
Notice that, in view of (1.2), (3.13) is consistent with 
\[\sum_{m,n\geq0}D(m,n)x^mq^n=\prod_{i=1}^\infty(1+xq^{2i-1}+x^2q^{4i-2})\]
which is Andrews' refinement of the Alladi-Schur theorem (namely, equation (1.1)) in generating function form.

\section*{Acknowledgment}
I would like to thank my doctoral advisor, Krishnaswami Alladi, for reading my manuscript and providing me with informed suggestions. Likewise, I thank George Andrews for his interest in the problem, which kept me determined to find a solution. Lastly, I sincerely appreciate the referee's helpful suggestions. \\


\bibliographystyle{amsplain}

\end{document}